\documentclass{article}
\usepackage{fullpage}
\usepackage{amsfonts, amsmath, amssymb,latexsym,epsfig}
\usepackage{amsthm}
\usepackage{tikz}
\usetikzlibrary{graphs}
\usetikzlibrary{graphs.standard}
\usepackage{relsize}
\usepackage{verbatim}
\usepackage{enumerate}
\usepackage[skip=10pt plus1pt, indent=20pt]{parskip}

\newtheorem{thm}{Theorem}[section]

\newtheorem{lem}[thm]{Lemma}

\newtheorem{prop}[thm]{Proposition}

\usepackage[english]{babel}

\usepackage[letterpaper,top=2cm,bottom=2cm,left=3cm,right=3cm,marginparwidth=1.75cm]{geometry}

\usepackage{amsmath}
\usepackage{graphicx}
\usepackage[colorlinks=true, allcolors=blue]{hyperref}

\title{The achromatic number of Paley graphs, Peisert graphs, and related Cayley graphs}
\author{Craig Timmons\thanks{Department of Mathematics and Statistics, California State University, Sacramento.   
		\texttt{craig.timmons@csus.edu}.  Research is supported by The Sacramento State Research and Creative Activity Faculty 
		Awards Program.}}

\begin{document}
\maketitle

\begin{abstract}
The achromatic number of a graph $G$, denoted $\chi_a (G)$, is the maximum number of colors in a proper vertex coloring of $G$.  Generalizing a result of Maistrelli and Penman, we prove that for sufficiently large $q$, the achromatic number of the $k$-th power Paley graph is at least $\frac{q-1}{k} + 1$.  This lower bound gives the correct order of magnitude in $q$, but determining the correct dependence on $k$ remains open for all $k \geq 3$.  The achromatic number of Peisert graphs is shown to be $\frac{q+1}{2}$.  These bounds are consequences of a general theorem involving Cayley graphs whose generating set is a union of cyclotomic classes.  Using character sums and Weil's bound, we prove an upper bound that improves the standard trivial bound $\binom{ \chi_a (G) }{2} \leq e(G)$.   
Finally, we consider random Cayley-sum graphs on an 
abelian group $\Gamma$ of order $n$, where elements are
included in the generating set independently with probability $\frac{1}{2}$.  We prove that 
the achromatic and pseudochromatic number of such a graph are both at most $\frac{(\sqrt{3}+o(1))n}{\sqrt{\log_2 n}}$ with high probability.
\end{abstract}

\section{Introduction}

A proper vertex coloring of a graph $G$ is an assignment of colors, one to each vertex, such that adjacent vertices are assigned different colors.  Such colorings play a central part of graph theory.  Variations and generalizations have been investigated for decades.  In 1967, Harary, Hedetniemi, and Prins \cite{HHP} introduced one such variation, called achromatic colorings, in their work on graph homomorphisms.  
Let $G$ be a graph.  A proper vertex coloring of $G$ is called \emph{achromatic} if there is at least one edge between each pair of distinct color classes.  The maximum number of colors in an achromatic coloring of $G$ is called the \emph{achromatic number of $G$} and is denoted $\chi_a (G)$.  Some authors refer to an achromatic coloring as a \emph{complete coloring}.  However, this term is also commonly used to refer to a vertex coloring (not necessarily proper) where there is at least one edge between any two distinct color classes.      
Let $\psi (G)$ denote the maximum number of colors in a complete
coloring of $G$, and call $\psi (G)$ the \emph{pseudochromatic number} of $G$.  It follows from the definitions that $\chi_a (G) \leq \psi (G)$ for every graph $G$.  
If $c$ is the number of colors in any complete coloring of $G$, then $\binom{c}{2} \leq e(G)$ because there is at least one edge between any two color classes.  Thus, 
\begin{equation}\label{ub:trivial1}
	\dbinom{ \chi_a(G) }{2} \leq e(G) ~~~\mbox{and}~~~
	\dbinom{ \psi (G) }{2} \leq e(G). 
\end{equation}
These inequalities can be rewritten as
\begin{equation}\label{ub:trivial2}
	\chi_a (G) \leq \sqrt{ 2 e(G) + \frac{1}{4}} + \frac{1}{2} 
	~~~ \mbox{and} ~~~
	\psi(G) \leq \sqrt{2 e(G) + \frac{1}{4}} + \frac{1}{2}.
\end{equation}
We call the bounds (\ref{ub:trivial1}) and (\ref{ub:trivial2}) the trivial (upper) bounds.  

Since its introduction in 1967, there has been much work on determining the achromatic number of various families of graphs.  One of the central questions is how close $\binom{\chi_a(G)}{2}$ is to the upper bound $e(G)$.  We will not attempt to survey all known results.  Instead, we list some of the more notable ones that concern this question.  Observe that equality in either of the bounds in (\ref{ub:trivial2}) is a stringent condition since it requires the number of edges of $G$ to be of the form $\frac{1}{2}c(c-1)$ for some integer $c \geq 2$.  Call a graph $G$ \emph{optimal} if equality holds in (\ref{ub:trivial2}) for $\chi_a(G)$.  Let 
$\{G_N \}$ be a sequence of graphs where $e(G_N)$ is a sequence tending to infinity.  Call this sequence of graphs
\emph{asymptotically optimal} if $\frac{ \chi_a(G_N)}{ \sqrt{2 e (G_N) }}  \rightarrow 1$ as $N \rightarrow \infty$.  If the sequence under consideration is clear from the context, we may say that $G_N$ is asymptotically optimal rather than $\{G_N \}$ and write $\chi_a(G_N) \approx \sqrt{2 e(G_N)}$.     

 For any $n \geq 2$, equality holds in (\ref{ub:trivial2}) when $G = K_n$.  Geller and Kronk \cite{GK} determined formulas for  $\chi_a(P_n)$ and $\chi_a(C_n)$ where $P_n$ and $C_n$ denote the path and cycle with $n$ vertices, respectively.  Their results show that  the sequence of paths and cycles are asymptotically optimal.  MacGillivray and Rodriguez \cite{MR} extended this to unions of paths and cycles.  Later, Cairnie and Edwards \cite{CE} proved a significant generalization to sequences of graphs with bounded degree.  They proved that for every sequence of graphs 
 $\{G_N \}$ with $\sup_N \Delta (G_N) < \infty$ and $e(G_N) \rightarrow \infty$, one has $\chi_a(G_N) \approx \sqrt{2 e(G_N)}$.  In the same paper, Cairnie and Edwards proved that almost all trees are asymptotically optimal in the following sense.  Given an integer $m$, let $q(m)$ be the largest integer $k$ such that $\binom{k}{2} \leq m$.  In \cite{CE}, it is proved that for almost all trees $T$ with $m$ edges, $\chi_a(T) = q(m)$.  By almost all trees with $m$ edges, it is meant that the probability a tree $T$ chosen uniformly at random from the set of all trees with $m$ edges satisfies $\chi_a(T) = q(m)$ with probability tending to 1 as $m$ tends to infinity.  McDiarmid \cite{M} studied the achromatic number of random graphs and proved that $\chi_a(G) = o ( \sqrt{2 e(G)})$ for almost all graphs; specifically $\frac{ \chi_a ( G_{n,1/2} ) }{ \sqrt{2 e( G_{n,1/2} ) }} = (1 + o(1)) \sqrt{ \frac{2}{ \log_2 n } }$ with probability tending to 1 as $n \rightarrow \infty$.   
 
 These results do not cover all that is known.  Nevertheless, one could argue that taken together, they suggest the achromatic number of a graph can be large when the graph has low maximum degree.  The theorem of Cairnie and Edwards \cite{CE} certainly supports this.  However, there are natural families of graphs which demonstrate that this condition is not necessary.  
 Let $\mathcal{Q}_d$ be the graph whose vertex set is the collection of all 0-1 sequences of length $d$ where two vertices are adjacent if they differ in only one coordinate.  The graph $\mathcal{Q}_d$ is often called the \emph{Hamming cube}.  Roichman \cite{Roichman} proved that $\left( \frac{1}{4} + o_d(1) \right) \sqrt{ e( \mathcal{Q}_d) } \leq \chi_a ( \mathcal{Q}_d ) \leq \sqrt{ 2 e ( \mathcal{Q}_d ) } + 1$.  Note that $\mathcal{Q}_d$ is $d$-regular and has $2^d$ vertices.  Hence, the degrees of the graphs in the sequence $\{ \mathcal{Q}_d \}_{d=1}^{\infty}$ are unbounded.  A similar lower bound on $\psi ( \mathcal{Q}_d )$ was proved a few years earlier by Blokhuis, Metsch, Moorhouse, Ahlswede, and Bezrukov,  \cite{ABBMM}.  
 Their focus was partitioning the vertices of the hypercube into sets with pairwise distance 1.  Noting that the property $\psi ( \mathcal{Q}_d ) > c \sqrt{2 e ( \mathcal{Q}_d )}$ for some constant $c > 0$ seemed rather special, they asked for other families of graphs for which this is true.  They noted it was false for almost all graphs, which by this time, had been proved by McDiarmid \cite{M}.  They also proposed a conjecture concerning Paley graphs which we define now.  
 
 Let $q \equiv 1 ( \textup{mod}~4)$ be a power of a prime and $\mathbb{F}_q$ be the finite field with $q$ elements.  The \emph{Paley graph}, denoted $G_2(q)$, is the graph with vertex set $\mathbb{F}_q$ and two vertices $x$ and $y$ are adjacent if $x-y=z^2$ for some $z \in \mathbb{F}_q \backslash \{0\}$.  In their influential paper on quasi-random graphs, Chung, Graham, and Wilson \cite{CGW} proved that Paley graphs form a family of quasi-random graphs.
 Perhaps motivated by this property and the achromatic number of a random graph, Blokhuis et al.\ \cite{ABBMM}
 conjectured that if $\{G_2(q) \}$ is a sequence of Paley graphs with $e(G_2(q)) \rightarrow \infty$, then 
 $\lim_{q \rightarrow \infty} \frac{ \psi (G_2(q))}{ \sqrt{ e( G_2(q))}}=0$.  
 
 In 2000, Thomason \cite{T} proved that the largest complete minor in $G_2(q)$ is $K_{ \frac{q+1}{2} }$.  This implies 
 $\psi ( G_2(q)) \geq \frac{q+1}{2}$ which disproves the 
 conjecture in \cite{ABBMM} since 
 $\sqrt{ e(G_2(q))} = \frac{1}{2} \sqrt{q^2 - q} \approx \frac{1}{2}q$.  In 2006, Maistrelli and Penman \cite{MP} used Thomason's argument to show $\chi_a( G_2(q)) = \frac{q+1}{2}$.  This suggests studying whether Thomason's argument extends to a broader family.  In this paper, we generalize the result of Maistrelli and Penman to a large family of Cayley graphs.  Two of these families are the $k$-th power Paley graphs and the Peisert graphs.
Let $k \geq 2$ be an integer and let $q$ be a power of a prime with $q \equiv 1 ( \textup{mod}~2k)$.  The \emph{$k$-th power Paley graph}, denoted $G_k(q)$, is the graph with vertex set $\mathbb{F}_q$ where $x$ and $y$ are adjacent if $x-y$ is a non-zero $k$-th power.

\begin{thm}\label{thm:k power Paley}
Let $k \geq 3$ be an integer and $q$ be a power of a prime with 
$q \equiv 1 ( \textup{mod}~2k)$.  For all $q$ sufficiently large in terms of $k$,   
	\[
	\frac{q-1}{k} +1 \leq \chi_a ( G_k (q) ) \leq 
	\psi ( G_k(q)) \leq 
	\sqrt{ \frac{q^2 - q}{k} + \frac{1}{4} } + \frac{1}{2}.
	\]
\end{thm}

Theorem \ref{thm:k power Paley} shows that $\chi_a ( G_k(q))$ is linear in $q$, but does not determine the dependence on $k$.  The coefficient of $q$ in the lower bound is $\frac{1}{k}$ and in the upper bound it is $\frac{1}{ \sqrt{k}}$.  For $k \geq 3$, finding the correct coefficient of $q$ is an open problem.    
      
Let $q$ be an even power of a prime $p$ where $p \equiv 3 ( \textup{mod}~4)$ and let $g$ be a primitive root of $\mathbb{F}_q$.  The \emph{Peisert graph}, denoted $P(q)$, is the graph with vertex set $\mathbb{F}_q$ where $x$ and $y$ are adjacent if $x-y = g^j$ for some $j \equiv 0,1 ( \textup{mod}~4)$.  Up to isomorphism, the definition of $P(q)$ does not depend on the primitive root.  These graphs are strongly regular,  self-complementary, and both vertex and edge-transitive.  They have appeared in many papers and their pseudo-randomness was established by Kisielewicz and Peisert \cite{KP}.   

\begin{thm}\label{thm:Peisert graph}
Let $q$ be an even power of a prime $p$ where $p \equiv 3 ( \textup{mod}~4)$.  If $q$ is sufficiently large, then 
\[
 \chi_a (P(q)) = \frac{q+1}{2}.
 \]
\end{thm}

The graph $G_k(q)$ and $P(q)$ are examples of 
Cayley graphs (as well as $\mathcal{Q}_d)$.  
Roichman \cite{Roichman} posed the problem of finding and characterizing families of Cayley graphs with large achromatic number.  Given a sequence of graphs $\{G_n \}_{n=1}^{\infty}$ with 
$e(G_n) \rightarrow \infty$, we say that this family has \emph{large achromatic number} if there is a constant $c  > 0$ such that $\chi_a (G_n) \geq c \sqrt{2 e(G_n) }$.  With the exception of random graphs, all of the families discussed so far have large achromatic number, and thus, large pseudochromatic number.
Taranchuk and the author \cite{TT} proved that certain polarity graphs have achromatic number as large as possible meaning equality holds in (\ref{ub:trivial1}).  Calhoun, Karabulut, Pigno, and the author \cite{Calhoun} established a lower bound on the achromatic number of certain unitary addition Cayley graphs, giving another algebraically defined family of graphs with large achromatic number.  To our knowledge, there is no explicit family of Cayley graphs $\{ G_n \}_{n=1}^{\infty}$ that satisfy the conjecture of \cite{ABBMM}.  If $G_t$ is the complete $t$-partite graph with $t$ vertices in each part, then $\chi (G_t) = \chi_a (G_t) = t$, $\psi (G_t) = \frac{t(t+1)}{2}$ \cite{Y}, and $e(G_t) = \binom{t}{2} t^2$.  Thus, $\frac{ \chi_a (G_t) }{  \sqrt{ 2 e(G_t)} } \rightarrow 0$ as $t \rightarrow \infty$ but $\frac{ \psi (G_t) }{ \sqrt{ 2 e (G_t)}} \rightarrow \frac{1}{2}$.  This demonstrates that requiring the classes to be independent adds a significant constraint.  

Theorems \ref{thm:k power Paley} and \ref{thm:Peisert graph} are consequences of a general result that applies to a large collection of Cayley graphs over finite fields.  For an integer $k \geq 2$, suppose that $q$ is a power of an odd prime such that $q-1 = 2 \alpha k$ for some $\alpha \in \mathbb{N}$.  Let $t $ be a positive divisor of $k$ with $t < k$ and write $k = t \beta$ (the case $t =k$ gives a complete graph $K_q$).  For a primitive root $g$ of $\mathbb{F}_q$, let 
\[
M_{k,t,g} = \cup_{i=0}^{t-1} g^i \langle g^k \rangle
= \left\{ g^{i + k j}: 0 \leq i \leq t-1 , j \in \mathbb{Z} \right\}.
\]
Call $s \in \mathbb{Z}_k$ \emph{admissible for $g$} if $g^{s + z} ( g^{\theta_2 t } - g^{\theta_1 t}) \notin M_{k,t,g}$ for all $0 \leq z \leq t-1$ and $1 \leq \theta_1 < \theta_2 \leq \beta$. 
Let $G_{k,t,g}(q)$ be the Cayley graph on $\mathbb{F}_q$ where $x$ is adjacent to $y$ if and only if $x-y \in M_{k,t,g}$.  
 The case when $t=1$ is the $k$-power Paley graphs, and the case when $k =4$ and $t =2$ is the case of the Peisert graphs.  For both of these families of graphs, changing the primitive root gives graphs that are the same up to isomorphism.      

\begin{thm}\label{thm:general lb}
	(a) If $g$ is a primitive root of $\mathbb{F}_q$ and $s$ is admissible for $g$, then 
	\[
	\frac{t(q-1)}{k} + 1 \leq \chi_a ( G_{k,t,g}(q))
	\leq \psi ( G_{k,t,g} (q) ) 
	\leq \sqrt{ \frac{ tq (q-1) }{k} + \frac{1}{4} } + \frac{1}{2}.  
	\] 
	
		\noindent
	(b) For sufficiently large $q$, in terms of $k$ and $t$, there is a primitive root $g$ of $\mathbb{F}_q$ such that 0 is admissible for $g$.  
	\end{thm}

The upper bounds given in Theorems \ref{thm:k power Paley} and \ref{thm:general lb} all follow from (\ref{ub:trivial1}).  Improving (\ref{ub:trivial1}) appears to be a challenging problem.  However, there are cases when improvements, sometimes rather small, can be made.  

\begin{thm}\label{thm:list of upper bounds}
	Let $k \geq 2$, $t$ be a positive divisor of $k$, and 
	$C_{k,t} = \displaystyle\max_{s \in \{1,2, \dots , k \} } \min \{ 1/s , 1 - (1 - t/k)^s \}$.  
	Suppose that $g$ is a primitive root for which an admissible $s$ exists.   
	
	\noindent
	(a) If $k \geq 2$ is even, then $\chi_a ( G_{k,k/2,g} (q) ) = \frac{q+1}{2}$.

\noindent
(b) If $k \geq 3$ is divisible by 3, then 
$\frac{q+2}{3} \leq \chi_a ( G_{k,k/3,g} (q) ) \leq \frac{q-1}{2}$.

\noindent
(c) If $1 \leq t \leq k/4$ where $k$ and $t$ are fixed integers, then  
\[
\chi_a ( G_{k , t , g} (q) ) \leq   C_{k,t} q + o(q).
\]
Moreover, as $t/k \rightarrow 0$, 
\[
C_{k,t} \leq \sqrt{ \frac{t}{k} } - \frac{t}{4k} + O((t/k)^{3/2}).
\]
\end{thm}

The proof of the upper bounds in parts (a) and (b) is a simple argument based on the maximum degree of $G_{k,t,g}(q)$.  Part (a) is the only case in which a formula for $\chi_a ( G_{k,t,g} (q))$ is known.  In all other cases, the lower and upper bounds differ by a constant in the leading term.  The proof of the upper bound in (c) makes use of character sums and Weil's bound.  The significance of (c) is that it improves upon (\ref{ub:trivial1}) (see the proof of Theorem \ref{thm:list of upper bounds} (c) in Section \ref{sec:3}).  Writing $C_{k,t}$ for the leading coefficient in the upper bound, it can be checked using Bernoulli's Inequality that 
$C_{k,t} <   \sqrt{t/k}$ for all $0 < t/k < 1$, whereas (\ref{ub:trivial1}) gives a leading term of $\sqrt{t/k}$.    

In the special case of cubic Paley graphs ($k=3$, $t =1$), for sufficiently large $q$ the bounds are $\frac{q+2}{3} \leq \chi_a ( G_3(q) ) \leq \frac{q-1}{2}$.  We conjecture that for large enough $q$, the lower bound is correct.  

Our final result shows that a large class of random Cayley-sum graphs have small achromatic number and pseudochromatic number.  Let $\Gamma$ be an abelian group and $A \subset \Gamma$.  Let $\textup{Cay}^+(A)$ be the graph with
vertex set $\Gamma$ and distinct vertices $x$ and $y$ are adjacent if $x+y \in A$.      
For a real number $0 < p < 1$, let $A$ be a random subset of $\Gamma$ where each element is included in $A$ with probability $p$, independent of all other elements in $\Gamma$.  The result is a random Cayley-sum graph $\textup{Cay}^+ (A)$ for some $A$.  Write $G_{ \Gamma , p}$ to represent a graph sampled from this probability space. 

\begin{thm}\label{thm:random Cayley-sum}
	Let $\mathcal{G}_n$ be the set of abelian groups of order $n$.  For each $\epsilon > 0$, 
	\[
\lim_{n \rightarrow \infty}	\sup_{ \Gamma \in \mathcal{G}_n }
	\mathbb{P} \left( \psi ( G_{ \Gamma , 1/2} ) > \dfrac{ ( \sqrt{3} + \epsilon )n }{ \sqrt{ \log_2 n } } \right) = 0.
		\]
	\end{thm}  
Since $\chi_a(G) \leq \psi (G)$ for all graphs $G$, Theorem \ref{thm:random Cayley-sum} also holds with $\chi_a$ in place of $\psi$.    
Using standard probabilistic arguments, it can be shown that $e(G_{ \Gamma , 1/2 } ) = (1+o(1))\frac{n^2}{4}$ with probability tending to 1 as $n$ tends to infinity.  Thus,
$\frac{ \psi  ( G_{ \Gamma , 1/2 } ) }{ \sqrt{ e( G_{ \Gamma ,  1/2 } ) }} \rightarrow 0$ as $n \rightarrow \infty$.  To our knowledge, it is still an open problem to find an explicit family of graphs $\{ G_n \}$ such that $\frac{ \psi (G_n)) }{ \sqrt{ e(G_n)}} \rightarrow 0$.  A related question on complete minors in pseudorandom graphs was posed by Thomason \cite{T}.


\section{Proof of Theorem \ref{thm:general lb}}\label{sec:2}

We will use the same notation introduced before the statement of Theorem \ref{thm:general lb}.  For convenience, we briefly state the assumptions again: $k \geq 2$ is an integer and $q$ a prime power with $q \equiv 1 ( \textup{mod}~2k)$, $q -1 = 2 \alpha k$, $k = t \beta$, $1 \leq t \leq \frac{k}{2}$, and $g$ is a primitive root of $\mathbb{F}_q$.  The condition $q \equiv 1 ( \textup{mod}~2k)$ ensures that $-1 \in \langle g^k \rangle$ for every primitive root $g$.  Finally, $s \in \mathbb{Z}_k$ is admissible for $g$ if for all $0 \leq r \leq t-1$ and $1 \leq \theta_1 < \theta_2 \leq \beta$, the element $g^{s+r}( g^{ \theta_2 t} - g^{\theta_1 t})$ does not belong to $M_{k,t,g}$.  

Let $G_{k,t,g}(q)$ be the Cayley graph with vertex set 
$\mathbb{F}_q$ where distinct vertices $x$ and $y$ are adjacent if $x-y \in M_{k,t,g}$.  The achromatic coloring we define is a natural generalization of
Maistrelli and Penman's coloring \cite{MP}, which they attribute to Thomason \cite{T}.  We will define the color classes now to serve as motivation for the lemmas to follow.  The coloring will depend on the choice of a primitive root $g$ in $\mathbb{F}_q$.  For the moment, suppose such a primitive root has been chosen.  For $0 \leq r \leq t-1$ and $0 \leq j \leq 2 \alpha - 1$, let $C_{r,j} = \{ g^{jk + s + r + \theta t } : 1 \leq \theta \leq \beta \}$, and $C_0 = \{ 0 \}$.  To prove that these classes give an achromatic coloring, we must show that there is an edge between any pair of classes, and each class is an independent set.  The latter is more involved because it requires showing that an admissible $s$ exists.  Hence, we first address the property that there is an edge between every pair of distinct classes.      

\begin{lem}\label{lem:translate Zq-1 lemma}
	In $\mathbb{Z}_{q-1}$, the sets $\delta + \{0,t , \dots , ( \beta - 1) t \} $ and $\{jk + i : 0 \leq i \leq t - 1 , j \in \mathbb{Z} \}$ have a nonempty intersection for all $\delta \in \mathbb{Z}_{q-1}$.
\end{lem}
\begin{proof}
	Let $\delta \in \mathbb{Z}_{q-1} = \{0,1, \dots , q-2 \}$.  Write $\delta = t \ell + i$ where $0 \leq i \leq t-1$.  Choose $w$ to be the least nonnegative residue of $- \ell ( \textup{mod}~ \beta)$ so that $\ell + w \equiv 0 ( \textup{mod}~\beta)$ and $w \in \{0,1, \dots , \beta -1 \}$. Then $\ell + w = j \beta$ for some $j \in \mathbb{Z}$.  Thus,
	\[
	\delta + wt = t \ell + i + w t =  t ( \ell+ w) +i = j \beta t + i= jk+i.
	\]
	Therefore,
	\[
	\delta + wt  \in ( \delta + \{ 0 , t , \dots , ( \beta - 1) t \} ) \cap \{ jk + i : 0 \leq i \leq t-1 , j \in \mathbb{Z} \}.
	\]
	\end{proof}  
   
   \begin{lem}\label{lem:edge between parts}
   	Let $0 \leq r_1 , r_2 \leq t-1$ and $0 \leq j_1 , j_2 \leq 2 \alpha -1$ with $(r_1,j_1) \neq (r_2,j_2)$.  There is at least one edge between $C_{r_1,j_1}$ and $C_{r_2,j_2}$, and at least one edge between $C_{r_1,j_1}$ and $C_0$.
   \end{lem}
   \begin{proof}
   	Let $g^{j_1 k + s + r_1 + t} - g^{j_2 k + s + r_2 + t} = g^{ \delta}$ where $\delta \in \mathbb{Z}_{q-1}$.  
   	The value $g^{\delta}$ is the difference of a vertex in $C_{r_1,j_1}$ and $C_{r_2,j_2}$.
   	and is nonzero since $(r_1,j_1) \neq (r_2 , j_2)$.   Using the definition of the $C_{r,j}$', the other differences obtained by using the same pair of classes are $g^{ \delta + t}, g^{ \delta + 2t}, \dots , g^{\delta + ( \beta - 1) t }$.  The exponents obtained in this way form the set $\delta + \{ 0 , t , 2t, \dots , ( \beta -1 )t \}$.  By Lemma \ref{lem:translate Zq-1 lemma}, there is a $w$, $i$, and integer $j$ with $i \in \{0,1, \dots , t-1 \}$, $0 \leq w \leq \beta-1$, and $jk + i \equiv \delta  + wt ( \textup{mod}~q-1)$.  Hence,
   	\[
   g^{j_1 k + s + r_1 + (w+1) t} - g^{j_2 k + s + r_2 + (w+1)t } = 	g^{ \delta + wt } = g^{j k + i } \in g^i \langle g^k \rangle 
   	\] 
   	 which implies there is an edge between $C_{r_1,j_1}$ and $C_{r_2,j_2}$.  Next, consider $C_{r_1,j_1}$ and $C_0$.  A similar argument, starting with the difference 
   	 \[
   	 g^{j_1 k  + s + r_1 + t } -0  := g^{\delta}
   	 \]
   	 shows that there is at least one edge between these two parts as well.  
   \end{proof}
   
   Lemma \ref{lem:edge between parts} is enough to prove a lower bound on $\psi( G_{k,t,g} (q))$ that does not require an assumption that $q$ is sufficiently large.    

\begin{lem}\label{lem:poly k power}
	If $s$ is admissible for $g$, then $C_{r,j}$ is an independent set for all $0 \leq r \leq t-1$ and $0 \leq j \leq 2 \alpha -1$.  
\end{lem}
\begin{proof}
	Let $g^{jk + s + r + \theta_1 t}$ and $g^{jk + s + r + \theta_2 t}$ be distinct vertices in $C_{r,j}$ with $1 \leq \theta_1 < \theta_2 \leq \beta$.  Now 
	\[
	g^{jk + s + r + \theta_2 t} - g^{jk + s + r  + \theta_1 t} 
	= 
	g^{jk + s + r} ( g^{\theta_2 t } - g^{\theta_1 t} ) \in M_{k,t,g}
	\]
	if and only if 
	\[
	g^{s + r }  ( g^{\theta_2 t } - g^{\theta_1 t} ) \in M_{k,t,g}.
	\]
	However, since $s$ is admissible for $g$, this latter statement does not hold so that $g^{jk + s + r + \theta_1 t}$ and $g^{jk + s + r + \theta_2 t}$ are not adjacent.  This shows that $C_{r,j}$ is an independent set.  
\end{proof}

\begin{proof}[Proof of Theorem \ref{thm:general lb} (a)]
	Suppose $g$ is a primitive root of $\mathbb{F}_q$ such that $s \in \mathbb{Z}_k$ is admissible for $g$. The sets $C_0 = \{0 \}$ and $C_{r,j} = \{ g^{jk + s + r +  \theta t } : 1 \leq \theta \leq \beta \}$ ($0 \leq r \leq t-1$, $0 \leq j \leq 2 \alpha - 1$) partition the vertex set of $G_{k,t,g}(q)$.  By Lemma \ref{lem:edge between parts}, there is an edge between any two distinct parts in this partition.  By Lemma \ref{lem:poly k power}, the parts are independent sets.  Since each $C_{r,j}$ has $\beta$ elements, the number of parts is $1 + \frac{q-1}{ \beta} = 1 + \frac{t(q-1)}{k}$.  This proves Theorem \ref{thm:general lb} (a).
	\end{proof}

To prove Theorem \ref{thm:general lb} (b), we show that for all large enough $q$, there is a $g$ for which an admissible $s$ exists, namely $s=0$.  We use the technique of Carlitz \cite{Carlitz} and it will be the most technical part of the proof.  The strategy will be to show that there is a primitive root $g$ of $\mathbb{F}_q$ such that $g^{dt} -1$ is a $k$-th power for all $1 \leq d \leq \beta - 1$.  Assuming this is the case for a moment, then in the product 
\[
g^r ( g^{ \theta_2 t } - g^{ \theta_1 t} ) = g^{\theta_1 t +r } 
( g^{ ( \theta_2 - \theta_1)t } - 1 ),
\]
the last factor will belong to $\langle g^k \rangle$.  
The first factor has an exponent of $r + \theta_1 t$ where $0 \leq r \leq t-1$ and $1 \leq \theta_1 \leq \beta -1$.  Therefore, 
\[
t \leq \theta_1 t + r \leq ( \beta-1)t + (t-1) = k - 1 
\]
which means $g^{r + \theta_1 t } ( g^{ ( \theta_2 - \theta_1)t } - 1 ) \notin \cup_{i=0}^{t-1} \langle g^k \rangle$.  We conclude that $C_{r,0}$ is an independent set for all $0 \leq r \leq t-1$.  
Since $C_{r,j} = g^{jk} C_{r,0}$ for all $1 \leq j \leq 2 \alpha -1$ and multiplication by $g^{jk}$ is an automorphism of 
$G_{k,t,g}(q)$, all of the $C_{r,j}$'s are independent sets.  Let us now prove that for sufficiently large $q$, the desired primitive root $g$ exists.  

\begin{lem}\label{lem:good primitive root}
	For sufficiently large $q$, there is a primitive root $g$ of $\mathbb{F}_q$ such that $g^{dt} -1$ is a $k$-th power for all $ 1 \leq d \leq \beta -1$.
	\end{lem}
	\begin{proof}
		Let $p = \textup{char} \mathbb{F}_q$.  Since $q \equiv 1 ( \textup{mod}~2k)$, $p$ does not divide $k$ and thus, $p$ does not divide $t$.  For $1 \leq d \leq \beta -1$, write $d = p^{a_d} e_d$ where $a_d$ is a nonnegative integer and $e_d$ is a positive integer with $p \nmid e_d$.  Let $\mathcal{E} = \{ e_d : 1 \leq d \leq \beta -1 \}$.  For $e \in \mathcal{E}$, let $F_e(X) = X^{et} - 1$.  Since $F_{e}' (X) = e t X^{et -1}$, the polynomials $F_e(X)$ and $F_{e}'(X)$ are relatively prime.  Therefore, $F_e(X)$ is square free.  This will be used when applying Weil's bound to a mixed character sum.  Let $Z = \{ x \in \mathbb{F}_q^* : F_e(x) = 0 ~\mbox{for some $e \in \mathcal{E}$} \}$.  For each $e \in \mathcal{E}$, $F_e(X) = X^{et}-1$ has at most $et$ roots so $|Z| \leq \sum_{ e \in \mathcal{E}} et :=D$.  Observe that $D$ depends only on $k$ and $t$.  
		
	The $k$-th nonzero powers in $\mathbb{F}_q$, denoted 
	$( \mathbb{F}_q^*)^k$, are equal to $\langle \gamma^k \rangle$ for any primitive root $\gamma$.  If $F_{e_d}(g) \in ( \mathbb{F}_q^*)^k$ for some $e_d \in \mathcal{E}$, then $g^{e_d t} -1 = \gamma^{ jk}$ for some $j \in \mathbb{Z}_{q-1}$.  This implies  
	\[
	\gamma^{ p^{a_d} j } = (g^{e_d t} - 1)^{p^{a_d} } = g^{p^{a_d} e_d t} - 1 = g^{dt} - 1. 
	\]	
	Hence, it is enough to prove that there is a primitive root $g$ such that $F_e(g) \in ( \mathbb{F}_q^*)^k$ for all $e \in \mathcal{E}$.  

Following the method of Carlitz \cite{Carlitz}, we define the appropriate indicator functions.  Since $k | q-1$, the cyclic group 
$\mathbb{F}_q^*$ has a multiplicative character of order $k$.  We fix a primitive root $\gamma$ and let $\psi$ be a primitive $k$-th root of unity.  Define $\chi : \mathbb{F}_q^* \rightarrow \langle \psi \rangle$ by $\chi ( \gamma ) = \psi$.  Extend $\chi$ to $\mathbb{F}_q$ by defining $\chi (0) = 0$.  For $e \in \mathcal{E}$, let 
\[
B_e (x) = \frac{1}{k} \left( 1 + \sum_{v=1}^{k-1} \chi ( F_e(x))^v \right).
\]
If $F_e(x) \neq 0$, then $B_e(x) = 1$ if $F_e(x) \in ( \mathbb{F}_q^*)^k$ and $B_e(x) = 0$ otherwise.  If $F_e(x)=0$, then $B_e(x) = \frac{1}{k}$.  Hence for all $x \notin Z$, $B_e(x)$ is the indicator function of a nonzero $k$-th powers.  To detect primitive roots, we use the standard primitive root indicator 
\[
P(x) = \frac{ \phi (q-1) }{q-1} \sum_{ \ell | q-1} \frac{ \mu  ( \ell )} { \phi ( \ell ) } \sum_{ \xi : \textup{ord}( \xi ) = \ell } \xi ( x)
\]
(see \cite{Carlitz}).  Here $\mu$ is the M\"{o}bius function and $\phi$ is the Euler phi function.  
Let 
\[
N_0 = \sum_{ x \in \mathbb{F}_q } P(x) \prod_{e \in E}B_e(x).
	\]
	Let $r = | \mathcal{E}|$ and $\mathcal{V} = \{ 0,1, \dots , k-1\}^r$.  We view the elements of $\mathcal{V}$ as vectors whose entries are indexed by the elements of $\mathcal{E}$.  With this notation, 
	\[
	\prod_{e \in \mathcal{E} } B_e(x) = \frac{1}{k^r} \sum_{ (v_1,\dots,v_r) \in \mathcal{V} } \chi \left( \prod_{i=1}^r F_{e_i}(x)^{v_i} \right) \]
	
	For $\mathbf{v} = (v_1, \dots , v_r) \in \mathcal{V}$, define the polynomial $G_{ \mathbf{v}} (X) = \prod_{e \in \mathcal{E} } F_e (X)^{v_e}$.  We now prove that $G_{ \mathbf{v}}$ is not a $k$-th power whenever $\mathbf{v}\neq \mathbf{0}$. Suppose $\mathbf{v}$ is such a vector and set $e_0 = \max \{ e \in \mathcal{E} : v_{e_{0}} \neq 0 \}$.  Let $\overline{\mathbb{F}_q}$ be an algebraic closure of $\mathbb{F}_q$ and $\tau$ be an element in 
	$\overline{\mathbb{F}_q}$ with multiplicative order $e_0 t$.  Thus, $\tau^{e_0 t} = 1$ and so $\tau$ is a root of $F_{e_0}(X)$.  Since $F_{e_0}(X)$ is square-free, $\tau$ is a root of multiplicity 1 and so contributes $v_{e_0}$ to its multiplicity as a root of $G_{ \mathbf{v}}$.  If $e \in \mathcal{E}$ and $\tau$ is a root of $F_e (X)$, then $\tau^{et} = 1$ so that $e_0 t | e t$ which gives $e_0 | e$.  By maximality, $e_0 = e$.  Therefore, the multiplicity of $\tau$ as a root of $G_{ \mathbf{v}} (X)$ is exactly $v_{e_0}$.  Since 
	$1 \leq v_{e_0} \leq k-1$, the multiplicity of $\tau$ is not divisible by $k$.  This shows $G_{ \mathbf{v} }(X)$ is not a $k$-th power in $\overline{\mathbb{F}_q}[X]$.  
	
	If $S( \xi , \mathbf{v} ) := \sum_{ x \in \mathbb{F}_q } \xi (x) \chi ( G_{ \mathbf{v} } (x) )$,  then $N_0$ can be written as  
	\[
	N_0 = \frac{ \phi ( q - 1) }{ (q - 1)k^r} \sum_{ \ell | q-1 } \frac{ \mu ( \ell ) }{ \phi ( \ell ) } 
	\sum_{ \xi : \textup{ord}( \xi ) = \ell } \sum_{ \mathbf{v} \in \mathcal{V} } S ( \xi , \mathbf{v} ).
	\]
	When $\mathbf{v} =  \mathbf{0}$, $G_{ \mathbf{v} }(X) = 1$ so 
	$S(\chi , \mathbf{0}) = \sum_{x \in \mathbb{F}_q } \xi (x)$.  This sum is 0 except in the case when $\xi = \chi_0$ where $\chi_0$ is the principal character ($\chi_0 (x) = 1$ for all $x \neq 0$).  The only contribution to $N_0$ occurs when $\ell = 1$ and we obtain  
	\[
	N_0 = \frac{ \phi (q-1) }{ ( q-1) k^r }  (q-1) +R= \frac{ \phi (q-1) }{k^r}+ R
	\]
	where $R$ is all remaining terms.  
Let
	\[
	N = N_0 - \sum_{x \in Z} P(x) \prod_{e \in \mathcal{E} } B_e (x).
	\]
	Since $0 \leq P(x), B_e(x) \leq 1$, the subtracted term is at most $|Z| \leq D$.  Thus,
	\[
	N \geq \frac{ \phi (q-1) }{k^r} - |R| - D.
	\]
	The next step is to bound $|R|$.  Let $\mathbf{v} \neq \mathbf{0}$ and $d( \mathbf{v} )$ be the number of distinct roots of $G_{\mathbf{v}}(X)$.  Then $d( \mathbf{v} ) \leq \sum_{e \in \mathcal{E} } et=D$ since $G_{ \mathbf{v}} (X) = \prod_{e \in \mathcal{E} } ( X^{et} - 1 )^{ v_e}$.  Since $G_{ \mathbf{v}} (X)$ is not a $k$-th power, Weil's bound for a mixed character sum implies there is a constant $C_{k,t}$, depending only on $k$ and $t$, such that 
	\[
	|S ( \xi , \mathbf{v} ) | \leq C_{k,t}  \sqrt{q} 
	\]
	for every character $\xi$.  Since there are $k^r-1$ nonzero vectors in $\mathcal{V}$, 
	\[
	|R| \leq \frac{ \phi (q-1) }{ (q-1)k^r } C_{k,t} \sqrt{q} 
	(k^r - 1) \sum_{ \ell | q - 1} \frac{ | \mu ( \ell ) |}{ \phi ( \ell) } | \{ \xi : \textup{ord} ( \xi ) = \ell \}|.
	\]
	The number of characters with order $\ell$ is $\phi (\ell)$ and it is known that $\sum_{ \ell | q-1 } | \mu ( \ell ) | = 2^{ \omega (q-1) }$ where $\omega (q-1)$ is the number of distinct prime factors of $q-1$.  Robin \cite{Robin} proved that $\omega (n) \leq  \frac{ 1.3841 \cdot \log n }{ \log \log n }$ so that $2^{ \omega (q-1)} = q^{ o(1)}$.   Thus,  
	\begin{equation*}
	|R| \leq \frac{ \phi (q-1) }{q-1} \left( 1 - \frac{1}{k^r} \right)  2^{ \omega (q-1) } C_{k,t} \sqrt{q} = q^{1/2 + o(1)}.   
	\end{equation*}
	Now $\phi (n) \geq \frac{ c n }{ \log \log n}$ for some constant $c >0$ and all large enough $n$ (see Theorem 15, \cite{Rosser}).  Hence,
	\[
	N  \geq \frac{ \phi (q-1) }{k^r }- |R| - D \geq C_{k,t}' \frac{ q}{ \log \log q} -  q^{1/2+o(1)} - D>0 
	\]
	where the last inequality holds for sufficiently large $q$.  	
	 Hence, there is a primitive root $g$ such that $F_e (g) \in ( \mathbb{F}_q^*)^k$ for all $e \in \mathcal{E}$.  As shown earlier, this implies $g^{dt - 1}$ is a $k$-th power for all $1 \leq d \leq \beta-1$.  
			\end{proof}    

\begin{lem}\label{lem:0 admissible}
	For sufficiently large $q$, there is a primitive root $g$ such that 0 is admissible for $g$.
	\end{lem}
	\begin{proof}
		By Lemma \ref{lem:good primitive root}, for large enough $q$ there is a primitive root $g$ such that $g^{dt -1}$ is a $k$-th power for $1 \leq d \leq \beta -1$.  Let $0 \leq r \leq t-1$ and $1 \leq \theta_1 < \theta_2 \leq \beta$.  We have 
		\[
		g^{0+r} ( g^{\theta_2 t } - g^{ \theta_1 t} ) = 
		g^{ r + \theta_1 t} ( g^{ ( \theta_2 - \theta_1) t } - 1 )  = 
		g^{r + \theta_1 t} g^{ jk}
		\]
		for some integer $j$. 
		 Now $0 + t \leq r + \theta_1 t \leq t - 1 +  ( \beta -1)t = k - 1$ so that $g^{r + \theta_1 t} \notin M_{k,t,g}$.  Hence, $g^{r + \theta_1 t} g^{ jk}$ is also not contained in $M_{k,t,g}$.  
	\end{proof}

\begin{proof}[Proof of Theorem \ref{thm:general lb}]
	By Lemma \ref{lem:0 admissible}, for large enough $q$ there is a primitive root $g$ such that 0 is admissible.  Consider the graph $G_{k,t,g}(q)$ and define the color classes $C_{r,j} = \{ g^{jk + s + r + \theta t} : 1 \leq \theta \leq \beta \}$ for $0 \leq r \leq t-1$ and $0 \leq j \leq 2 \alpha -1 $.  The element 0 is in its own class, denoted $C_0$.  By Lemmas \ref{lem:edge between parts} and \ref{lem:poly k power}, these classes define an achromatic coloring of $G_{k,t,g}(q)$ with $\frac{ t(q-1) }{k} + 1$ colors.  
	\end{proof}


\section{Proof of Theorem \ref{thm:list of upper bounds}}\label{sec:3}

In this section, we use two different counting arguments to prove upper bounds on the achromatic number of a graph.  Both bounds hold for arbitrary graphs.  The first bound is an elementary observation.   

\begin{prop}\label{prop:singleton class ub}
	If $G$ is an $n$-vertex graph with maximum degree $\Delta (G)$, then 
	\[\psi (G) \leq \max \left\{ \Delta (G) + 1 , \left\lfloor \frac{n}{2} \right\rfloor \right\}.
	\] 
\end{prop}
\begin{proof}
	Suppose $G$ has a pseudochromatic coloring with $\kappa$ colors.  If there is a color class with only one vertex, then that vertex must be incident to at least one vertex in each of the remaining color classes.  Thus, $\kappa \leq \Delta (G) +1$.  If all color classes have at least two vertices, then $\kappa \leq \lfloor \frac{n}{2} \rfloor$.  
\end{proof}

\begin{proof}[Proof of Theorem \ref{thm:list of upper bounds} (a) and (b)]
	Assume that $g$ is a primitive root for which an admissible $s$ exists.  The graph $G_{k,k/2,g}(q)$ is a regular graph of degree $\frac{ t(q-1) }{k} = \frac{q-1}{2}$.  
	By Proposition \ref{prop:singleton class ub}, 
	\[
	\chi_a(G_{k,k/2,g}(q)) \leq \psi ( G_{k,k/2,g}(q) ) \leq \frac{q+1}{2}.
	\]  
	The matching lower bound is given by Theorem \ref{thm:general lb}.  For part (b), assume that 3 divides $k$ and consider the graph $G_{k , k/3 , g}(q)$.  This graph is regular of degree 
	$\frac{ t(q-1) }{k} = \frac{q-1}{3}$ so by Proposition  
	\ref{prop:singleton class ub}, 
	\[
	\chi_a(G_{k,k/3,g}(q)) \leq \psi ( G_{k,k/3,g}(q) ) \leq \left\lfloor \dfrac{q}{2} \right\rfloor  = \dfrac{q-1}{2}.
	\]  
	Again, the lower bound is given by Theorem \ref{thm:general lb}. 
\end{proof}

For a graph $G$, write $\overline{G}$ for the complement 
of $G$.  For a vertex $v \in V(G)$, let $N_G (v)$ be the set of vertices that are adjacent to $v$.  For any integer $s \geq 1$, let $\mathcal{I}_s(G)$ be the collection of all independent sets in $G$ with $s$ vertices. Define 
\[
\gamma_s (G) = \min_{S \in \mathcal{I}_s(G) } \left| \displaystyle\bigcap_{v \in S} N_{ \overline{G} } (v) \right|. 
\]
The value $\left| \cap_{v \in S} N_{ \overline{G} } (v) \right|$ is the number of vertices in $G-S$ that are not adjacent (in $G$) to any vertex in $S$.  

\begin{prop}\label{prop:general upper bound}
Let $G$ be an $n$-vertex graph.  If $G$ has an achromatic coloring using $\kappa$ colors and $s$ is the number of vertices in a smallest color class, then    
\[
\kappa \leq \min \left\{ \left\lfloor \frac{n}{s} \right\rfloor , n - s - \gamma_s (G) + 1 \right\}. 
\]
\end{prop}
\begin{proof}
	Suppose we have an achromatic coloring of $G$ with $\kappa$ colors.  Let $C$ be a color class with $|C| = s$.  Then $\kappa s \leq n $ so that $\kappa \leq \lfloor \frac{n}{s} \rfloor$.  Also, every color class not equal to $C$ must have a vertex adjacent to a vertex in $C$.  Thus,
	\begin{eqnarray*}
		\kappa -1 & \leq & \left| \cup_{v \in C} N_G (v) \right| = n - s - \left| \cap_{v \in C} N_{ \overline{G} }(v) \right| \\
		& \leq & n - s - \gamma_s (G)
	\end{eqnarray*}
	which gives $\kappa \leq n -s - \gamma_s (G) + 1$.  
\end{proof}

The next step is to establish a lower bound for $\gamma_s( G_{k,t,g}(q) )$ so that Proposition \ref{prop:general upper bound} can be applied to this family of graphs.  The edge density of $G_{k,t,g}(q)$ is approximately $\frac{t}{k}$ and so on average, given a fixed set $C$ of $s$ vertices, $G_{k,t,g}(q)$ has about $q \left( 1 - \frac{t}{k} \right)^s$ that are not adjacent to any vertex in $C$.  This is proved in the following lemma.   

\begin{lem}\label{lem:common nhood}
	For every $s$ with $1 \leq s \leq k$, 
	\[
	\gamma_s  ( G_{ k,t,g}(q)) \geq  q \left( 1 - \frac{t}{k} \right)^s - O_{k,t}( \sqrt{q}) .
	\]
\end{lem}
\begin{proof}
	For convenience, we write $G$ for the graph $G_{k,t,g}(q)$ and $M$ for $M_{k,t,g}$ throughout this proof.  
	 Let $\zeta$ be a primitive $k$-th root of unity and $\chi : \mathbb{F}_q^* \rightarrow \langle \zeta \rangle$ be the multiplicative character $\chi (g) = \zeta$.  The order of $\chi$ is $k$ and we define $\chi (0)=0$.  Let $B(x)$ be the indicator for the set $M$.  Since $0 \notin M$, we have $B(0) = 0$ and for $x \neq 0$, a formula for $B(x)$ is 
	\[
	B(x) = \frac{1}{k} \sum_{m=0}^{k-1} \left( \sum_{ i = 0}^{t-1} \zeta^{-mi} \right) \chi^m (x). 
	\]
	
	Let $A_m = \frac{1}{k} \sum_{i=0}^{t-1} \zeta^{ - m i}$ and  $L = \sum_{m=1}^{k-1} |A_m|$.  
	With this notation, for every $x \in \mathbb{F}_q^*$, 
	\[
	B(x) = \frac{t}{k} + \sum_{m=1}^{k-1} A_m \chi^m (x). 
	\]
	  This expansion will be used only for $x-c$ where $x \notin J$ and $c \in J$ so that $x-c \neq 0$.  
	  
	Let $C$ be an independent set in $G$ with $|C| = s$.  Define $\eta (C)  = \left| \bigcap_{c \in C} N_{ \overline{G} }(c) \right|$.  
	For $\emptyset \neq J \subseteq C$, let 
	\[
	\nu(J) = \left| \bigcap_{c \in J } (c  + M) \right|. 
	\]
	Since $C$ is an independent set, no vertex in $C$ is contained in the union $\bigcup_{c \in J} ( c+ M)$.  Thus,
	\[
	\eta (C) = q - s  - \left| \bigcup_{c \in C} ( c+ M ) \right| = 
	 q - s - \sum_{ \emptyset \neq J \subseteq C} (-1)^{ |J| + 1} \nu (J).
	 \]
	Next we prove that $\nu (J) \approx q(t/k)^{|J|}$.  Suppose that $J = \{c_1,c_2,\dots , c_r \}$ for some $1 \leq r \leq s$.  Since $0 \notin M$, the intersection $\bigcap_{c \in J} (c+M)$ does not contain an element of $J$ so that
	\[
	\nu (J) = \sum_{ x \notin J} \prod_{c \in J} B(x-c). 
	\]
	Since $x-c$ is not zero when $x \notin J$ and $c \in J$, 
	\[
	\nu (J) = \sum_{ x \notin J} \prod_{c \in J} B(x-c) = \sum_{x \notin J} \prod_{c \in J} \left( \frac{t}{k} + \sum_{m=1}^{k-1} A_m \chi^m ( x -c) \right).
	\]
	One term in the expanded form of this product is $(q-r) \left( \frac{t}{k} \right)^{ |J| } = (q-r) \left( \frac{t}{k} \right)^r$.  Any other term is indexed by some nonempty set $I \subseteq J$ and integers $m_c \in \{1,2, \dots , k - 1 \}$ where $c \in I$.  Suppose we have such an index set $I$ with $\ell$ elements.
	Without loss of generality, say $I = \{c_1, \dots , c_{ \ell } \}$.  The absolute value of the corresponding term is at most 
	\begin{equation*}
		r (t/k)^{r- \ell} L^{ \ell}  + L^{ \ell} (t/k)^{ r- \ell} \left| \sum_{x \in \mathbb{F}_q} \chi \left( \prod_{j=1}^{ \ell} (x-c_j)^{m_j} \right) \right| 
			 \leq r( t/k)^{r - \ell} L^{ \ell} + L^{ \ell} (t/k)^{ r-\ell} ( \ell - 1) \sqrt{q} \leq
			 2 L^{ \ell} r \sqrt{q}.
				\end{equation*}
	For the first inequality, a Weil bound is being applied to the polynomial $\prod_{j=1}^{\ell} ( X- a_j)^{m_j}$.  This polynomial is not a $k$-th power since $1 \leq m_j \leq k-1$ for all $j$.   
	Thus, 
	\begin{equation*}
	| \nu (J) - q (t/k)^r |  \leq    r (t/k)^r + 
	\sum_{ \ell = 1}^{r} \dbinom{r}{ \ell} 2 L^{ \ell} r \sqrt{q} 
	\leq 3 r \sqrt{q}(1 + L)^r  
	\end{equation*}
	so that $\nu(J) = q ( t/k)^{ |J|}  + \epsilon_J$ where $|\epsilon_J| \leq 3 |J| (1 + L)^{ |J|} \sqrt{q}$.  With this estimate, we have 
	\begin{eqnarray*}
		\eta (C) & = & q - s - \sum_{ \emptyset \neq J \subseteq C} (-1)^{ |J| + 1} \nu (J) \\		
		&\geq & q - s +q \sum_{r=1}^{s} \dbinom{s}{r} (-1)^{r }  q ( t/k)^r - \sum_{r=1}^{s} \dbinom{s}{r} 3r(1+L)^r \sqrt{q} \\
		& \geq & q (1 - t/k)^s - s -  3 s \sqrt{q} \sum_{r=0}^{s} \dbinom{s}{r}  (1+L)^r  \\
		& = & 	q ( 1 - t/k)^s - s - 3 s \sqrt{q}  ( 2 +L)^s \\
		& \geq & q ( 1-t/k)^s - 4s \sqrt{q} ( 2 + L)^s.
		\end{eqnarray*}
		The last step in proving a lower bound for $\eta (C)$ is to prove an upper bound on $L$.  Since $A_m = \frac{1}{k} \cdot \frac{ 1 - \zeta^{ - mt} }{1 - \zeta^{-m} }$, we have 
		\[
		|A_m| \leq \frac{ | \sin ( \pi m t /k ) | }{ k | \sin (\pi m /k ) | }.
		\]
		For $1 \leq m \leq \frac{k}{2}$, $|A_m| \leq \min \{ t/k , 1/2m \}$ and by symmetry, $|A_m| = |A_{k - m } |$.  Hence,
		\[
		L \leq 2 \sum_{m=1}^{ \lfloor k/2 \rfloor } \min \{ t/k , 1/2m \} \leq 2 + \log t.
		\]
		Thus, 
		\[
		\eta( C) \geq q ( 1 - t/k)^s - 4s \sqrt{q} (4 + \log t)^s \geq q (1-t/k)^s - 4s \sqrt{q} (4 + \log t)^k = q(1-t/k)^s - O_{k,t}(\sqrt{q}).
		\]   Since $C$ was an arbitrary independent set with $s $ vertices and $\eta (C) \geq \gamma_{s}(G)$, 
		\[
		\gamma_s( G_{k,t,g}(q) ) \geq q \left( 1 - \frac{t}{k} \right)^s - O_{k,t}( \sqrt{q}).  
			\]
	\end{proof}

\begin{lem}\label{lem:general Ckt ub}
	Let $F_{k,t}(s) = \min \left\{ \frac{1}{s} , (1 - t/k)^s \right\}$ and $C_{k,t} = \max_{1 \leq s \leq k} F_{k,t}(s)$.  For sufficiently large $q$, 
	\[
	\chi_a ( G_{ k,t,g}(q)) \leq \max \left\{ \left\lfloor \frac{q}{k+1} \right\rfloor , ( C_{k,t} + o(1)) q.  \right\}
	\]
	For sufficiently large $q$ where an admissible $s$ for $g$ exists in $\mathbb{F}_q$, 
	\[
	\chi_a ( G_{ k,t,g}(q)) \leq  ( C_{k,t} + o_{k,t}(1)) q.  
	\]
\end{lem}
\begin{proof}
	Consider an achromatic coloring of $G_{k,t,g}(q)$ with $\kappa$ colors. 
	Let $C$ be a smallest color class and assume $s = |C|$.  If $s \geq k+1$, then $\kappa s \leq q$ so that $\kappa \leq \lfloor \frac{q}{k+1} \rfloor$.  If $s \leq k$, then Proposition \ref{prop:general upper bound} applied to this class gives 
	  $\kappa  \leq \min\{ \lfloor q/s \rfloor ,q - s - \gamma_s ( G_{k,t,g}(q) ) + 1 \}$.  By Lemma \ref{lem:common nhood}, 
	  $\gamma_s ( G_{k,t,g} (q) ) \geq q(1-t/k)^s - O_{k,t} ( \sqrt{q})$.  Therefore, 
	 \[
	 \kappa \leq \min \{ q/s , q ( 1 -(1- t/k)^s ) + O_{k,t} ( \sqrt{q} ) \}
	 \leq ( F_{k,t}(s) + o_{k,t}(1)) q \leq ( C_{k,t} + o_{k,t}(1))q.
	 \] 
	 In either case,
	 \[
	 \kappa \leq \max \left\{ \left\lfloor \frac{q}{k+1} \right\rfloor ,  (C_{k,t} + o_{k,t}(1) )q \right\}.
	 \] 
	 For the second assertion, if an admissible $s$ exists, then by Theorem \ref{thm:general lb}, $\chi_a (G_{k,t,g}(q)) \geq \frac{ t(q-1)}{k} + 1$. For fixed $k$ and $t$, this is larger than $\lfloor \frac{q}{k+1} \rfloor$ for large enough $q$. Thus, an optimal achromatic coloring cannot have a smallest color class of size at least $k+1$ so the upper bound $\lfloor \frac{q}{k+1} \rfloor$ does not apply.  
\end{proof}

	\begin{proof}[Proof of Theorem \ref{thm:list of upper bounds} (c)]
		Suppose an admissible $s$ for $g$ exists.  The first upper bound on $\chi_{a} (G_{k,t,g}(q))$ 
		follows from Lemma \ref{lem:general Ckt ub}, but we will also prove that this upper bound gives a coefficient of $q$ that is less than $\sqrt{ \frac{t}{k} }$ which is the coefficient of $q$ in (\ref{ub:trivial1}).  To achieve this, we show that $F_{k,t}(s) < \sqrt{ \frac{t}{k}}$ for all $1 \leq s \leq k$.  For convenience, let $\rho = \frac{t}{k}$.  If $s > \frac{1}{ \sqrt{\rho}}$, then $F_{k,t}(s) \leq \dfrac{1}{2} < \sqrt{ \rho}$.  Assume $s \leq \frac{1}{ \sqrt{\rho}}$.  By Bernoulli's inequality, $F_{k,t}(s) \leq 1 - (1 - \rho)^s \leq s \rho \leq \sqrt{\rho}$.  This last inequality is strict when $s < \frac{1}{ \sqrt{\rho}}$.  If $s = \frac{1}{ \sqrt{\rho}}$, then $1 - (1 - \rho)^s < s \rho$.  Since $F_{k,t}(s) < \sqrt{ \rho }$ for all $1 \leq s \leq k$, we have $C_{k,t} < \sqrt{ \rho}$.  
		
		For the upper bound on $C_{k,t}$ when $t/k \rightarrow 0$,  we must show that if $\rho \rightarrow 0$, then $C_{k,t} \leq \sqrt{ \rho} - t/4k + O ( (t/k)^{3/2})$.  
		For real $x > 0$, let $f(x) = \frac{1}{x}$ and $g(x) = 1 - (1- \rho )^x$.  Then $f(x)$ is decreasing, $g(x)$ is increasing, $f(1) =1$, $g(1) = \rho < 1$, $\lim_{x \rightarrow \infty} f(x) = 0$ and $\lim_{x \rightarrow \infty} g(x) = 1$.  Thus, $f(x) = g(x)$ has a unique positive solution.   
		Since $C_{k,t} = \max_{ s \in \{1,2, \dots , k \} } \min \{ 1/s , 1-( 1 - \rho)^s \}$, the maximum value of $C_{k,t}$ is at most $\frac{1}{s}$ where $s$ is the positive real number satisfying $f(s) = g(s)$.    
Let $a=-\log(1-\rho)$ and $y=1/s$.  The equation $\frac{1}{s} = 1 - (1- \rho)^s$ is equivalent to $y=1-e^{-a/y}$.  Taking logarithms and simplifying gives $a=-y\log(1-y)$.  Using the Maclaurin expansion on the right and simplifying leads to $a=y^2+\frac{y^3}{2}+O(y^4)$.
Set $\alpha =\sqrt a$ and take square roots to get 
\begin{equation}\label{tech eq}
\alpha=y\sqrt{1+\frac y2+O(y^2)}
=y+\frac{y^2}{4}+O(y^3).
\end{equation}
This equation implies $y= \alpha +O( \alpha ^2)$. Substituting back into (\ref{tech eq}) gives
\[
y= \alpha -\frac{ \alpha^2}{4}+O( \alpha^3)
=\sqrt a-\frac{a}{4}+O(a^{3/2}).
\]
Finally, \(a=\rho+O(\rho^2)\).  Therefore, $\frac{1}{s}=y
=\sqrt{\rho}-\frac{\rho}{4}+O(\rho^{3/2}) = \sqrt{t/k} - \frac{t}{4k} +  O ((t/k)^{3/2})$.  Through this calculation, $s$ has been treated as a positive real number.  Since the maximum that determines $C_{k,t}$ is taken over $\{1,2, \dots,  k  \}$, the upper bound still applies.      
\end{proof}

\section{Proof of Theorem \ref{thm:random Cayley-sum}}\label{sec:4}

In this section we prove Theorem \ref{thm:random Cayley-sum}.  We begin with three lemmas.  

\begin{lem}\label{lem:generic bound on color class sizes}
	Let $G$ be an $n$-vertex graph and $C$ and $a$ be positive constants with $Ca > 1$.  If $V(G)$ is partitioned into $k$ parts where $k \geq \frac{Cn}{ \sqrt{ \log_2 n}}$, then there are at least $ \frac{(C - \frac{1}{a})n}{ \sqrt{ \log_2 n}}$ parts of size at most 
	$\lceil a \sqrt{ \log_2 n } \rceil$.
\end{lem}
\begin{proof}
	Let $t = \lceil a \sqrt{ \log_2 n } \rceil$.  If $s$ is the number of parts of size at most $t$ in a partition that satisfies the hypothesis, then $n \geq (k - s)t$.  
	Solving for $s$ and using $k \geq \frac{Cn}{ \sqrt{ \log_2 n }}$ and $t \geq a \sqrt{ \log_2 n }$ gives 
	\[
	s \geq k - \frac{n}{t} \geq \frac{Cn}{ \sqrt{ \log_2 }} - \frac{n}{ a \sqrt{ \log_2 n } } = \left( C - \frac{1}{a} \right) \frac{n }{ \sqrt{ \log_2 n }}.
	\]
\end{proof}

\begin{lem}\label{lem:sum rainbow extraction}
	Let $s$, $t$, and $n$ be positive integers with $s \geq t \sqrt{8n}$.  Let $C_1 , \dots , C_s$ be pairwise disjoint subsets of an abelian group $\Gamma$ of order $n$.  If $|C_i| \leq t$ for all $i$, then there is a set $I \subseteq \{1,2, \dots , s \}$ such that $|I|  \geq \frac{s^{2/3} }{ 4n^{1/3}t^{2/3} }$
	and the collection of sumsets $C_i+C_j$, $\{i,j\} \in \dbinom{I}{2} $ is pairwise disjoint.
\end{lem}
\begin{proof}
	Let $p  = (nst^2)^{-1/3}$.  Let $J$ be a random subset of $[s]$ where each $i \in [s]$ is included in $J$ with probability $p$, independently of all other elements in $[s]$.  We then have $\mathbb{E}(|J|) = ps = \frac{s^{2/3}}{n^{1/3}t^{2/3}}$.  
	
	The next step will be to delete indices from $J$ so that we obtain a subset $I \subseteq J$ that satisfies the desired properties.  If a sumset $C_i + C_j $ intersects a sumset $C_{i'} + C_{j'}$, then there is an element $a \in \Gamma$ such that 
	$a = \alpha_i + \beta_j = \alpha_{i'} + \beta_{j'}$ and 
	$\alpha_i \in C_i$, $\beta_j \in C_j$, $\alpha_{i'} \in C_{i'}$, $\beta_{j'} \in C_{j'}$.  We define a graph that 
	records such intersections.
	For each $a \in \Gamma$, let $G_a$ be the graph with vertex set $\{1,2, \dots , s \}$ and $\{i,j \}$ is an edge if and only if $a \in C_i + C_j$.  We claim that $\Delta (G_a) \leq t$.  Fix a vertex $i$ and let $j_1, \dots , j_M$ be the neighbors of $i$ in $G_a$.  
	For each $\alpha \in C_i$, there is at most one $\beta \in \Gamma$ such that $a = \alpha + \beta$ since $a - \alpha = x$ has a unique solution for $x$.  As the sets $C_1, \dots , C_s$ are pairwise disjoint, such a $\beta$ can only belong to one $C_j$.  Therefore, the degree of $i$ is at most $t$ since $|C_i| \leq t$. We conclude $\Delta (G_a) \leq t$ and in particular, $G_a$ has at most $\frac{1}{2}st$ edges.
	
	Let $Y$ be the number of index pairs $\{i,j\} \neq \{i',j' \}$ such that $a \in C_i + C_j$ and $a \in C_{i'} + C_{j'}$ for some $a \in \Gamma$.  Equivalently, $Y$ is the number of pairs of edges $\{e , e' \}$ that are contained in the same graph $G_a$ for some $a$.  Therefore, since $\Delta (G_a) \leq t$ and $e(G_a) \leq \frac{1}{2} st$, 
	\[
	\mathbb{E}(Y) \leq n \left(  p^3 s t^2 + p^4 \left( \frac{1}{2}st \right)^2 \right)=
	1 + \frac{s^{2/3}}{4n^{1/3}t^{2/3}}.
	\]
	Thus, 
	\[
	\mathbb{E}( |J| - Y ) \geq ps - \left( 1 + \frac{s^{2/3}}{4n^{1/3}t^{2/3}} \right) 
	= \dfrac{3s^{2/3}}{4 n^{1/3} t^{2/3} } - 1
	\geq \dfrac{s^{2/3}}{4n^{1/3}t^{2/3}}
	\]
	where the latter inequality holds since $s \geq t \sqrt{8n}$.  Choose an index set $J$ such that $|J|  - Y \geq \frac{s^{2/3}}{4n^{1/3}t^{2/3}}$.  For each conflicting pair of index pairs counted by $Y$, delete one index from either pair to obtain a set $I \subset J$ such that $|I| \geq \frac{s^{2/3}}{4n^{1/3}t^{2/3}}$ and the sumsets $C_i + C_j$, $\{i,j \} \in \dbinom{I}{2}$ are pairwise disjoint.  
\end{proof}

\begin{lem}\label{lem:probability for extracted family}
	Suppose $D_1 , \dots , D_r$ are disjoint subsets of an abelian group $\Gamma$ of order $n$ such that $|D_i| \leq t$, and the sumsets $D_i + D_j$ are all pairwise disjoint.  Then in the random Cayley sum graph $G_{ \Gamma , 1/2}$,  
	\[
	\mathbb{P}( e(D_i , D_j) > 0 \mbox{~for all~}1 \leq i < j \leq r )
	\leq \textup{exp} \left( -2^{-t^2} \dbinom{r}{2} \right).
	\]
\end{lem}
\begin{proof}
	First note that since each $D_i$ has at most $t$ elements, $|D_i + D_j| \leq t^2$ for all $1 \leq i < j \leq r$.  Hence, the probability of $\{ A \cap (D_i + D_j) = \emptyset \}$ is at least $2^{-t^2}$ so that the probability of $\{ e(D_i,D_j) > 0 \}$ is at most $1 - 2^{-t^2}$.    
	Since the sumsets $D_i + D_j$ are pairwise disjoint, if $\{i,j \} \neq \{i' , j' \}$, the events $\{ e(D_i,D_j) > 0 \}$ and 
	$\{ e(D_{i'} , D_{j'}) \}$ are independent.  Therefore, 
	\[
	\mathbb{P}( e(D_i , D_j) > 0 \mbox{~for all~}1 \leq i < j \leq r )
	\leq (1  - 2^{-t^2})^{ \binom{r}{2} } \leq \mbox{exp} \left( -2^{-t^2} \dbinom{r}{2} \right).
	\]
\end{proof}

\begin{proof}[Proof of Theorem \ref{thm:random Cayley-sum}]
	Suppose there is a complete coloring of $G_{ \Gamma , 1/2}$ with at least $\frac{Cn}{ \sqrt{\log_2 n }}$ colors for some constant $C > \sqrt{3}$.  Choose $a$ so that $a < \frac{1}{ \sqrt{3} }$ and $Ca > 1$.  By Lemma \ref{lem:generic bound on color class sizes}, there are at least $s:= \left\lfloor \frac{ (C - \frac{1}{a})n}{ \sqrt{ \log_2 n}} \right\rfloor$ color classes of size at most $t:=\lceil a \sqrt{ \log_2 n} \rceil$. 
	By Lemma \ref{lem:sum rainbow extraction},  
	for sufficiently large $n$ there is a set of at least $r = \left\lfloor  \frac{s^{2/3}}{4n^{1/3}t^{2/3}} \right\rfloor $ color classes whose sumsets are pairwise disjoint.  By Lemma \ref{lem:probability for extracted family}, for any fixed collection of such color classes, the probability the coloring restricted to these classes is complete is at most 
	$\textup{exp}( - 2^{-t^2} \binom{r}{2} )$.   The total number of fixed collections is at most $(n^t)^r = \textup{exp}(  rt \log n)$.  Therefore, the probability that $G_{ \Gamma , 1/2}$ has a complete coloring with at least $\frac{Cn}{ \sqrt{ \log_2 n }}$ colors is at most 
	\begin{equation}\label{eq:exponential}
	\textup{exp}\left( rt \log_2 n - 2^{-t^2} \binom{r}{2} \right).
	\end{equation}
	We have $r  = \Theta \left( \frac{ n^{1/3} }{ ( \log_2 n )^{2/3} } \right)$ so that $rt \log_2 n = O \left(  n^{1/3} ( \log_2 n)^{5/6} \right)$.
	For the negative exponent in (\ref{eq:exponential}), 
	\[
	2^{-t^2} \dbinom{r}{2} = n^{- a^2 +o(1) } \cdot  \Theta \left( \dfrac{n^{2/3} }{ ( \log_2 n)^{4/3} }\right) = n^{-a^2 + 2/3 + o(1)} .
	\]
	Since $a^2 < \frac{1}{3}$, there is a $\delta >0$ depending only on $a$ and $C$ such that $n^{-a^2 + 2/3 + o(1))} = n^{ 1/3  + \delta +o(1)}$.        
	This, together with $rt \log_2 n = O \left( n^{1/3}  ( \log_2 n)^{5/6} \right)$ implies that the exponent in  (\ref{eq:exponential}) tends to $- \infty$ as $n \rightarrow \infty$.  Therefore, the probability that $G_{ \Gamma , 1/2}$ has a complete coloring with at least $\dfrac{Cn}{ \sqrt{ \log_2 n } }$ colors tends to 0 as $n \rightarrow \infty$.    
\end{proof}

\subsection{Use of AI}  
ChatGPT 5.6 was used to proofread several iterations of this paper.  The first draft was prepared by the author.  With the exception of Lemmas \ref{lem:good primitive root} and \ref{lem:common nhood}, all of the ideas and proofs in Sections \ref{sec:2} and \ref{sec:3} are due to the author.  ChatGPT 5.6 simplified the author's original proofs of Lemma \ref{lem:translate Zq-1 lemma} and Proposition \ref{prop:singleton class ub} and also suggested the term \emph{admissible}.  Section \ref{sec:4} was written by the author, but all of the ideas came from prompting ChatGPT 5.6 to consider random Cayley-sum graphs and the paper of McDiarmid \cite{M}. 



\end{document}